\documentclass{article}
\usepackage[text={7in, 9in}, centering, letterpaper]{geometry}
\usepackage{amsthm}
\usepackage{amsmath}
\usepackage{amssymb}
\usepackage{amsfonts}
\usepackage{enumerate}
\usepackage{subfigure}
\usepackage{tikz}
\usepackage{graphicx}
\usepackage{hyperref}

\newtheorem{thm}{Theorem}
\newtheorem{lemma}[thm]{Lemma}

\newtheorem{mydef}[thm]{Definition}

\newcommand{\SP}{\theta}
\newcommand{\ASP}{\text{off}}
\newcommand{\R}{\mathbb{R}}
\newcommand{\PONT}{\text{Align}}
\newcommand{\Pt}{\text{Al}}
\newcommand{\LL}{\mathcal{L}}
\newcommand{\SL}{\mathcal{SL}}

\title{Optimally Perturbed Identity Matrices of Rank 2}
\date{\today}

\author{Robi Bhattacharjee \\ \href{mailto:rcbhatta@eng.ucsd.edu}{rcbhatta@eng.ucsd.edu}}

\begin{document}

\maketitle

\abstract{
The problem of optimal antipodal codes can be framed as finding low rank Gram matrices $G$ with $G_{ii} = 1$ and $|G_{ij}| \leq \epsilon$ for $1 \leq i \neq j \leq n$. In 2018, Bukh and Cox introduced a new bounding technique by removing the condition that $G$ be a gram matrix. In this work, we investigate how tight this relaxation is, and find exact results for real valued matrices of rank $2$. 
}

\section{Introduction}

An antipodal code is a set of pairs of points $(x, -x)$ on the unit sphere in $\R^d$ in which the minimum distance between distinct points is maximized. We can identify each pair $(x, -x)$ with a single unit vector $u_x$. The distance between two unit vectors corresponds uniquely to their dot product. Therefore, we can consider the problem of finding $n$ unit vectors on $S^{d-1}$ such that no two have a large dot product by absolute value. This is generally studied by considering Gram matrices, giving the following optimization problem. 

\begin{mydef}\label{spdef}
Let $\SP(n,d)$ denote the minimum $\epsilon$ such that there exists an $n \times n$ Gram matrix $B$ of rank at most $d$ with $B_{ii} = 1$ for all $i$, and $|B_{ij}| \leq \epsilon$ for all $i \neq j$. 
\end{mydef}

While this problem has been most successfully studied through the linear programming bounds pioneered by Delsarte and Geothals \cite{del}, Bukh and Cox \cite{bukhCox} have found success by taking a fundamentally different approach. They ignore the constraint that $B$ be a Gram (equivalently symmetric positive semidefinite) matrix, and propose the following relaxation.

\begin{mydef}\label{aspdef}
Let $\ASP(n,d)$ denote the minimum $\epsilon$ such that there exists an $n \times n$ matrix $B$ of rank at most $d$ with $B_{ii} = 1$ for all $i$, and $|B_{ij}| \leq \epsilon$ for all $i \neq j$.
\end{mydef}

They then observe that $\ASP(n,d) \leq \SP(n,d)$, and proceed to find lower bounds of $\ASP(n,d)$. By finding equality cases with these bounds, they are able to exactly compute $\SP(n,d)$ for several previously unknown cases.

This suggests the natural question of understanding exactly how tight the inequality $\ASP(n,d) \leq \SP(n,d)$ is. In particular, are there any cases in which $\ASP(n,d) < \SP(n,d)$, or is this inequality always an equality? Answering this question would not only characterize the strength of the techniques used by Bukh and Cox, but would also be of great importance in computing $\ASP(n,d)$. Bounding $\ASP(n,d)$ has a number of applications including uses in coding theory, derandomization, and low dimensional embeddings (see \cite{alon} for example), and therefore investigating the relationship between $\ASP(n,d)$ and $\SP(n,d)$ could have benefits for understanding both of them. 

The goal of this paper is to make further steps towards understanding the relationship between these two problems. Our main results are as follows.

First, the central parameter used by Bukh and Cox to bound $\ASP(n,d)$ is the following. 

\begin{mydef}
\cite{bukhCox} Let $\mu$ be a nonzero probability mass on $\R^k$ and define $$\LL(\mu) = \inf_{y \in \text{supp}(\mu)\setminus \{0\}} \inf_{v \in \R^k \setminus \{0\}} \frac{\mathbb{E}_{x \sim \mu} [|\langle v,x\rangle|}{|\langle v,y\rangle|}.$$ Let $\mathcal{P}(n, k)$ be the collection of all probability masses $\mu$ for which there is a multiset $X$ of $n$ vectors over $\R^k$ that span $\R^k$ and where $\mu$ is the uniform distribution over $X$. Then $$\mathcal{SL}(n, k) = \sup_{\mu \in \mathcal{P}(n,k)} \LL(\mu).$$
\end{mydef}

They then show that $$\ASP(n,d) \geq \frac{1}{n\SL(n,n-d) - 1},$$ and subsequently focus on finding upper bounds of $\SL(n, n-d)$. We tighten this inequality to an equality.

\begin{thm}\label{thm_equal_improv}
For all $n, d$, $$\ASP(n,d) = \frac{1}{n\SL(n,n-d) - 1}.$$
\end{thm}

 Our other main result is about matrices of rank $2$.

\begin{thm}\label{thm_dim_2}
For all $n$, $\ASP(n, 2) = \SP(n,2) = \cos \frac{\pi}{n}$. 
\end{thm}

 Previously known cases of equality include all $n, d$ such that $$\ASP(n,d) = \SP(n,d) = \sqrt{\frac{n-d}{ d(n-1)}}.$$ This case corresponds to the Welch bound \cite{welch}. Bukh and Cox showed that for all $n, d$ with $n-d \in \{1, 2, 3, 7, 23\}$, $\ASP(n,d) = \SP(n,d)$. 
 
 What these cases have in common is that they all result from some construction of a symmetric positive definite matrix meeting a lower bound of $\ASP(n,d)$. In particular, the Welch bound can be generalized to apply to bound $\ASP(n,d)$, and the techniques employed by Bukh and Cox specifically bound $\ASP(n,d)$. Simply put, all previously known cases of equality occurred when $\ASP(n,d)$ was \textit{used} to bound $\SP(n,d)$. 
 
This is in contrast to the case $d=2$ which is presented in this paper, as bounding and computing $\SP(n, 2)$ is quite trivial whereas bounding $\ASP(n,2)$ is considerably more challenging.  Because of this, we believe that our result suggests a deeper relationship between $\ASP(n,d)$ and $\SP(n,d)$.

In our proof, we use an indirect method where we show how to ``symmetrize" any"locally optimal" matrix of rank $2$ while preserving the minimal entry $|M_{ij}|$. We leave open the conjecture that $\ASP(n,d) = \SP(n,d)$ for all $n,d$, and hope that our techniques for $d=2$ lead to further headway on this problem. 

\section{Acknowledgments}
We thank Henry Cohn for his insightful advice and overall supervision of this project. We also thank Anson Kahng and Yuchen Fu for their helpful comments and suggestions in writing this paper. Finally, we thank the anonymous referee for their suggestions on connecting this work to the existing literature as well as their suggestions for improving the presentation of this paper.

\section{Preliminaries}

We begin by presenting an alternative definition of $\SL(n,k)$. Broadly speaking, $\SL(n,k)$ is defined with respect to distributions of $n$ points over $\R^k$. We will instead consider subspaces of rank $k$ in $\R^n$. To do so, we introduce a quantity called ``alignment".

\subsection{Alignment}

\begin{mydef}
For any $a \in \R^n \backslash \{0\}$ we define $$\Pt_i(a) = \frac{|a_i|}{\sum_{j \neq i} |a_j|}.$$ In the case that $\sum_{j \neq i} |a_j| = 0$, we define $\Pt_i(a) = \infty$. Here $\Pt_i$ stands for the alignment in the $i$th index.
\end{mydef}

Clearly, $\Pt_i(a)$ is highly dependent on which basis $a$ is expressed in. Therefore, we always assume $\R^n$ to have a fixed orthonormal basis $e_1, e_2, \dots, e_n$ for which $\Pt_i$ is defined. We define the alignment of a space to just be the maximum alignment of any vector in it. 

\begin{mydef}
Given a subspace $A \subset \R^n$, we let $$\Pt_i(A) = \max_{a \in A \setminus \{0\}} \Pt_i(a).$$ The alignment of $A$ (denoted $\Pt(A)$) is defined as $$\Pt(A) = \max_{i \in \{1, 2, \ldots, n\}} \Pt_i(A).$$ Finally, we let $\PONT(n, k)$ denote the minimum value of $\Pt(A)$ over all subspaces of dimension at least $n-d$. 
\end{mydef}

Note that the alignment of a space is well defined because $\Pt_i(a)$ is a scale invariant function, and because the set of unit vectors forms a compact set. 

The relationship between alignment and $\SL(n,k)$ is as follows. Let $\mu$ be a uniform distribution over $n$ non-zero vectors, $x_1, x_2, \dots, x_n \in \R^k$. For any $v \in \R^k$ let $v_x = (\langle v, x_1 \rangle, \langle v, x_2 \rangle, \dots, \langle v, x_n \rangle).$ Then for any $v \in \R^k$, and $1 \leq i \leq n$, we have
\begin{equation*}
\begin{split}
\frac{\mathbb{E}_{x \sim \mu} [|\langle v,x\rangle|}{|\langle v,x_i\rangle|} &= \frac{\sum_{j=1}^n |\langle v, x_j \rangle |}{n |\langle v, x_i \rangle|} \\
&= \frac{1}{n} + \frac{\sum_{j \neq i} |\langle v, x_j \rangle |}{n |\langle v, x_i \rangle|} \\
&= \frac{1 + \frac{1}{\Pt_i(v_x)}}{n}.
\end{split}
\end{equation*}

Building from this observation, we have the following.

\begin{thm}\label{thm_def_equiv}
For all $n \geq k > 0$, $$\frac{1 + \frac{1}{\PONT(n, k)}}{n} = \SL(n, k).$$
\end{thm}

\begin{proof}
Consider any $\mu \in \mathcal{P}(n,k)$, and let it be the uniform distribution over $\{x_i| 1 \leq i \leq n\} \subset \R^k$. Fix an orthonormal basis $e_1, e_2, \dots e_n$ of $\R^n$, and let $T:\R^k \to \R^n$ be the map $$T(v) = \sum \langle v, x_i \rangle e_i.$$ For any $v \in \R^k \setminus \{0\}$, we have 
\begin{equation*}
\begin{split}
\frac{\mathbb{E}_{x \sim \mu} [|\langle v,x\rangle|}{|\langle v,x_i\rangle|} &= \frac{\sum_{j=1}^n |\langle v, x_j \rangle |}{n |\langle v, x_i \rangle|} \\
&= \frac{1}{n} + \frac{\sum_{j \neq i} |\langle  v, x_j \rangle |}{n |\langle v, x_i \rangle|} \\
&= \frac{1 + \frac{1}{\Pt_i(T(v))}}{n}.
\end{split}
\end{equation*}
Let $A$ be the image of $T$. It follows that $$\LL(\mu) = \inf_{v \in \R^k} \frac{\mathbb{E}_{x \sim \mu} [|\langle v,x\rangle|}{|\langle v,x_i\rangle|} = \frac{1 + \frac{1}{\Pt(A)}}{n} \leq \frac{1 + \frac{1}{\PONT(n, k)}}{n}.$$ Because $x_i$ spans $\R^k$, we can also verify that $A$ has dimension $k$. 

To finish the proof, it suffices to show that some corresponding distribution $\mu$ exists for every $A \subset \R^n$ of rank $k$. To do this, simply pick $k$ vectors $a_1, a_2, \dots, a_k$ that span $A$. Let $\mu$ be the uniform distribution over the $n$ rows of $[a_1, a_2, \dots, a_k]$.  Then $$\SL(\mu) \geq \LL(\mu) = \frac{1 + \frac{1}{\Pt(A)}}{n}.$$ This implies the result. 
\end{proof}

Next, we further characterize the relationship between $\ASP(n,d)$ and alignment with the following theorem. 

\begin{thm}\label{duality}
Let $A \subset \R^n$ be a vector space with dimension $n-d$ such that $\Pt_i(A)$ is finite for all $i$. Then there exists a matrix $G$ with the following properties.
\begin{enumerate}
	\item $G$ has rank at most $d$;
	\item $G_{ii} = 1$ for $1 \leq i \leq n$;
	\item $|G_{ij}| \leq \max_{a \in A} \Pt_i(a)$ for $i \neq j$;
	\item $Ga = 0$ for all $a \in A$. 
\end{enumerate}
\end{thm}

\begin{proof}

Our strategy is to construct $v_1, v_2, \dots, v_n$ such that for all $1 \leq i \leq n$ the following three conditions hold. $v_i$ is orthogonal to all $a \in A$, $(v_i)_i = 1$, and $|(v_i)_j| \leq \Pt_i(A)$ for $i \neq j$. The result then follows from taking $G$ to be the matrix with rows $v_i$. Without loss of generality, set $i =1$.

Let $\rho: A \to \R^{n-1}$ be the map sending the first coordinate of any vector to 0. Consider $a = (a_1, a_2, \dots, a_n) \in A$. If $\rho(a) = 0$, then $a_i = 0$ for $i > 1$. Since $\Pt_1(A)$ is finite, it follows that $a_1 = 0$ as well. Therefore, $\rho$ has a trivial kernel. 

As a result, we can define a linear map, $f: \rho(A) \to \R$ as $f(\rho(a)) = a_1$ where $\rho(A)$ denotes the image of $A$ under $\rho$. We now use a lemma based on the Hahn Banach theorem (a proof can be found in the appendix).

\begin{lemma}\label{holder}
Let $V \subset \R^m$ be a subspace and $f: V \to \R$ be a linear map. Fix some orthonormal basis of $\R^m$ so that all $v \in V$ can be represented with coordinates $v_1, v_2, ..., v_m$. Let $\epsilon$ be the smallest positive real such that for all $v \in V$ $$f(v) \leq \epsilon \sum |v_i|.$$ Then there exists $u \in \R^m$ such that $|u_i| \leq \epsilon$ for all $i$ and such that for all $v \in V$, $f(v) = \langle u, v \rangle.$
\end{lemma}

Taking $m = n-1$ and $V = \rho(A)$, Lemma \ref{holder} yields a vector $r \in \R^{n-1}$ such that $f(\rho(a)) = \langle r, \rho(a) \rangle$. Let $r' \in \R^n$ denote the unique vector such that $r'_1 = 0$ and $\rho(r') = r$. Then for any $a \in A$,
\begin{equation*}
\begin{split}
\langle e_1, a \rangle &= a_1 \\
&= f(\rho(a)) \\
&= \langle r, \rho(a) \rangle \\
&= \langle r', a \rangle. 
\end{split}
\end{equation*}
Therefore, we have $\langle a, e_1 - r' \rangle = 0.$ By construction, $r$ has all elements of absolute value at most $\epsilon$. Therefore taking $v_1 = e_1 - r'$ suffices. An analogous process can be used to find vectors $v_2, v_3, \dots, v_n$ which complete the proof.
\end{proof}

This theorem implies that $\ASP(n,d) \leq \PONT(n,n-d)$. This gives a proof of Theorem \ref{thm_equal_improv}.

\begin{proof}
(Of Theorem \ref{thm_equal_improv}) As shown by Bukh and Cox, $\ASP(n,d) \geq \frac{1}{n\SL(n,n-d) - 1},$  and as shown above, $\ASP(n,d) \leq \PONT(n, n-d)$. Therefore, by Theorem \ref{thm_def_equiv}, we must have equality between all three quantities.
\end{proof}

\subsection{An alternative formulation}

Recall that $\SP(n,d)$ is the smallest real number $\epsilon$ such that there exist $n$ vectors in $\R^d$ with all pairwise dot products having absolute value at most $\epsilon$. We now present a similar characterization of $\ASP(n,d)$, which will prove useful for proving Theorem \ref{thm_dim_2}.

For any finite set of vectors $S = \{s_1, s_2, \dots, s_n\} \subset \R^d \setminus \{0\}$, let $$H(S) = \{ \sum_1^n \lambda_is_i; \sum |\lambda_i| \leq 1\}.$$ In other words, $H(S)$ is the convex region spanned by all $s_i \in S$ as well as their negatives. Next, let $S_i = S \setminus \{s_i\}$. Then we define the parameter $\alpha_i$ as $$\alpha_i(S) = \max_{rs_i \in H(S_i)} r.$$ $\alpha_i(S)$ can be thought of as the fraction of $s_i$ that is contained in $H(S_i)$. 

\begin{thm}\label{geopont}
For any $S \subset \R^d \setminus \{0\}$ with $|S| = n$, let $A \subset \R^n$ denote the space of all $a$ such that $\sum a_is_i = 0$. Then $\alpha_i(S) = \Pt_i(A)$. 
\end{thm}

\begin{proof}
We first show that $\Pt_i(A) \leq \alpha_i(S)$. Fix $a \in A$, and let $\beta$ denote $\Pt_i(a)$. Then $$\beta= \frac{|a_i|}{\sum_{j \neq i} |a_j|}.$$ By the definition of $A$, we also have that $a_is_i = -\sum_{j \neq i}a_js_j.$ Dividing both sides by $\sum_{j \neq i} |a_j|$, we see that $$\beta s_i = \sum_{j \neq i} \frac{a_j}{\sum_{j \neq i} |a_j|} s_j,$$ which implies that $\beta w_i \in H(S_i)$. Since $a$ was arbitrary, it follows that $\Pt_i(A) \leq \alpha_i$.

Next, we show that $\Pt_i(A) \geq \alpha_i(S)$. By the definition of $\alpha_i$, there exists some $\lambda$ such that $\sum_{j \neq i} |\lambda_j| = 1$ and $\alpha_i(S) s_i = \sum_{j \neq i} \lambda_j s_j.$ Let $a_i = \alpha_i(S)$, and $a_j = -\lambda_j$ for $i \neq j$. Then $a = (a_1, a_2, \dots a_n)$ is an element of $A$ that satisfies $\Pt_i(a) = \alpha_i$, thus $\Pt_i(A) \geq \alpha_i$. 
\end{proof}

\begin{thm}\label{geopontcor}
Let $S_{n,d}$ denote all sets of $n$ non-zero vectors in $\R^d$. Then $$\ASP(n,d) = \min_{S \in S_{n,d}} \max_{1 \leq i \leq n} \alpha_i(S).$$
\end{thm}

\begin{proof}
By Theorem \ref{geopont}, $\min_{S \in S_{n,d}} \max_{1 \leq i \leq n} \alpha_i(S) \geq \PONT(n, n-d)$. It remains to show that equality holds. Let $A$ be a space such that $\Pt(A) = \PONT(n, n-d)$. Let $W$ be a matrix with at most $d$ rows and with $n$ columns such that $W$'s rows span the orthogonal complement of $A$. Letting $S$ be the set of $W$'s columns suffices. 
\end{proof}
\emph{ }\\ We now summarize our three formulations of $\ASP(n,d)$. If $\epsilon = \ASP(n,d)$, then we must have,
\begin{itemize}
	\item Some matrix $G$ of rank $\leq d$ with $G_{ii} = 1$ and $|G_{ij}| \leq \epsilon$,
	\item Some space $A \subset \R^n$ of rank at least $n-d$ with $\Pt(A) = \epsilon$,
	\item Some set $S = \{s_1, s_2, \dots, s_n\} \subset \R^d \setminus \{0\}$ with $\max_i \alpha_i(S) = \epsilon$. 
\end{itemize}
All three of these ``interpretations" will play a role in proving Theorem \ref{thm_dim_2}.
\section{Proof of Theorem \ref{thm_dim_2}} \label{d2}

We now show that $\ASP(n,2) = \SP(n,2) = \cos \frac{\pi}{n}$ for all $n \geq 2$. Let $\ASP(n,2) = \epsilon$. Since $\ASP(n,2) \leq SP(n,2)$, we know that $\epsilon \leq \cos \frac{\pi}{n} < 1$. Applying our results on $\ASP(n,d)$, we define the following:
\begin{itemize}
	\item Let $Q$ be a matrix of rank $\leq 2$ with $Q_{ii} = 1$ and $|Q_{ij}| \leq \epsilon$. 
	\item Let $A \subset \R^n$ be a subspace with rank at least $n-2$ such that $\Pt(A) = \epsilon$.
	\item Let $S = \{w_1, w_2, \dots w_n \} \subset \R^2 \setminus \{0\}$ be a set of non-zero vectors such that $\alpha_i(S) \leq \epsilon$ for $1 \leq i \leq n$.
\end{itemize}
Based on Theorems \ref{duality} and \ref{geopont}, we make the following additional assumptions and definitions.
\begin{itemize}
	\item For each $i$, let $\epsilon_i$ denote $\Pt_i(P)$. Thus $\epsilon = \max_ i \epsilon_i$.
	\item $A$ is the null space of $Q^t$, and is all the space of all linear dependencies of $w_i$. In particular, $\sum_1^n a_iw_i = 0$ if and only if $(a_1, a_2, \dots, a_n) \in A$. 
	\item $|Q_{ij}| \leq \epsilon_j$ with equality holding for some $i$ for each $j$. 
	\item Let $P$ denote a matrix with column space $A$ (implying $Q^tP = 0$) such that $P_{ii} = 1$, and $\Pt_i(p_i) = \epsilon_i$, where $p_i$ denotes the $i$th column of $P$. 
	\item $\alpha_i(S) = \epsilon_i$. 
	\item Let $\theta_i$ denote the directed angle from the $x$-axis to $w_i$. We assume $0 \leq \theta_1 \leq \theta_2 \leq \dots \leq \theta_n \leq \pi$. In other words, the $w_i$ are ordered by their angles with the $x$-axis, and all have positive $y$-coordinate (this can be guaranteed by noting the symmetry between $w_i$ and $-w_i$). 
\end{itemize}
The main steps of the proof are as follows:

\begin{enumerate}
	\item We utilize $w_1, w_2, \dots, w_n$ to derive several properties about the structure of $P$.
	\item We leverage $Q$'s local optimality to construct a diagonal matrix $\Lambda$ such that $Q^tP\Lambda = P\Lambda Q^t = 0$.
	\item We show that $P\Lambda$ is symmetric and conclude that $\frac{Q + Q^t}{2}$ is a gram matrix of rank $2$.
\end{enumerate}

\begin{lemma}\label{uniquep}
For $1 \leq i \leq n$, $p_i$ is the unique vector up to scale in $A$ such that $\Pt_i(p_i) = \epsilon_i$. Furthermore, there exist real numbers $a_1, b_1, a_2, b_2, \dots a_n, b_n$ with $a_i, b_i < 0$ for $1 \leq i \leq n-1$, and $a_n, b_n > 0$ such that 
\[P = \left( \begin{array}{cccc}
1 & a_1 & & b_n\\
b_1 & \ddots & \ddots & \\
& \ddots & \ddots & a_{n-1} \\
a_n & & b_{n-1} & 1 \end{array} \right),\]
with all other terms in the matrix being 0.
\end{lemma}

\begin{figure}[h]
\centering
\begin{tikzpicture}
\draw (-2.8, 0) -- (2.8,0);
\node at (2.9, 0.1) {$w_1$};
\node at (-2.9, -0.2) {$-w_1$};

\draw (2.5,1) -- (-2.5,-1);
\node at (2.6, 1.1) {$w_2$};
\node at (-2.6, -1.2) {$-w_2$};

\draw (1, 2.35) -- (-1, -2.35);
\node at (1.1, 2.45) {$w_3$};
\node at (-1.1, -2.55) {$-w_3$};

\node at (-0.5, 1) {$\dots$};
\node at (0.5, -1) {$\dots$};

\draw (-2.5, 1.3) -- (2.5, -1.3);
\node at (-2.4, 1.4) {$w_n$};
\node at (2.9, -1.2) {$-w_n$};

\draw [green] (0,0) -- (2.5,0);
\node [green] at (1.6, -0.5) {$\epsilon_1w_1$};

\draw [red] (2.5,1) --(1, 2.35)-- (-2.5, 1.3) --(-2.5,-1)--(-1, -2.35)--(2.5, -1.3)-- (2.5,1);
\node [red] at (-1, 2.15) {$H(S_1)$};
\end{tikzpicture}
\caption{$S$ and $H(S_1)$}\label{fig1}
\end{figure}
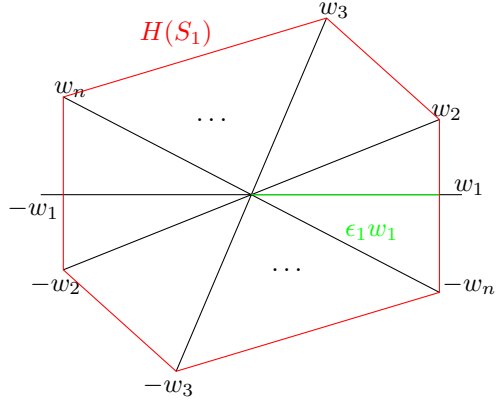

\begin{proof}
We first show that $w_1, w_2, \dots, w_n, -w_1, -w_2, \dots -w_n$ form a convex polygon (in that order). Assume they didn't. Then $w_i \in H(S_i)$ for some $i$ which in turn implies that $\epsilon \geq \epsilon_i \geq 1$ a contradiction. We can similarly show that no two vectors $w_i, w_j$ are parallel. This also implies that $H(S_i)$ is the union of a convex polygon with its interior.

Recall that $p_i$ is the $i$th column of $P$ such that $p_{ii}= 1$ and $\Pt_i(p_i) = \epsilon_i$. Since $\sum a_iw_i = 0$ for all $a = (a_1, a_2, \dots a_n) \in A$, it follows that  $w_i = -\sum_{j \neq i} (p_i)_jw_j.$ Let $\lambda_j = \epsilon_i (p_i)_j$. Then $\sum_{j \neq i} |\lambda_j| = 1$, and $\epsilon_i w_i = \sum_{j \neq i} \lambda_j w_j.$

By the maximality of $\epsilon_i$, $\epsilon_i w_i$ must be on the boundary of $H(S_i)$, and consequently on the line segment $\overline{w_{i-1}w_{i+1}}$ (see Figure \ref{fig1} for an illustration with $i = 1$). Therefore, $|\lambda_{i-1}| + |\lambda_{i+1}| = 1$, and $\lambda_j = 0$ for $j \notin \{i-1, i+1\}$. This means $(p_i)_j = 0$ for $j \notin \{i-1, i, i+1\}$ which implies that $P$ must have the desired form. 

Because $w_{i-1}, w_i$, and $w_{i+1}$ are non-parallel vectors, there is a unique linear combination of $w_{i-1}$ and $w_{i+1}$ that yields $\epsilon_i w_i$. This linear combination corresponds to the unique vector (up to scale) in $P$ with alignment $\epsilon_i$. Because $(p_i)_i = 1$, $p_i$ is uniquely determined. Finally, since $w_i$ is on the line segment $\overline{w_{i-1}w_{i+1}}$, this linear combination must have strictly positive coefficients, and this implies that $a_i, b_i$ have the desired signs. 

\end{proof}

\begin{lemma}\label{uniform}
$\epsilon_1 = \epsilon_2 = \dots = \epsilon_n = \epsilon$.
\end{lemma}

\begin{proof}
As demonstrated in the proof of Lemma \ref{uniquep}, $\epsilon_iw_i$ is the intersection of $\overline{w_{i-1}w_{i+1}}$ with $w_i$. 

Suppose we pick some sufficiently small $\delta > 0$ and replace $w_i$ with $w_i(1-\delta)$. Because $w_1, w_2, \dots w_n$ are the vertices of a convex polygon, this will strictly increase $\epsilon_i$ and strictly decrease $\epsilon_{i+1}$ and $\epsilon_{i-1}$. All other $\epsilon_j$ will be unchanged. Therefore, if there exists $i$ with $\epsilon_i < \epsilon_{i +1}$, we can apply this operation and strictly decrease $\epsilon_{i+1}$. Repeatedly applying this, in any case in which not all $\epsilon_i$ are equal, we can decrease the largest one contradicting the assumption that $\epsilon$ is optimal. Therefore, all $\epsilon_i$ are equal. 
\end{proof} \emph{ } \\
For the next several lemmas, we use the following abuse of notation.
\begin{enumerate}
	\item For any matrix $B$ we will use $\Pt(B)$ to denote $\Pt(C(B))$ where $C(B)$ denotes the column space of $B$.
	\item We similarly use $\Pt_i(B)$ to denote $\Pt_i(B) = \max_{b \in C(B)} \Pt_i(b).$
\end{enumerate}

\begin{lemma}\label{opbounds}
Let $M$ be a matrix. Then for $1 \leq i \leq n$,
\begin{enumerate}
	\item If $\langle q_i, Mp_i \rangle < 0$, then as $x \to 0$, $\Pt_i(P + xMP) \leq \epsilon -\Theta(x)$.
	\item If $\langle q_i, Mp_i \rangle = 0$, then as $x \to 0$, $\Pt_i(P + xMP) \leq \epsilon + o(x)$.
\end{enumerate}
\end{lemma}

\begin{proof}
Fix $i$. Let $\hat{P}$ denote the set of all unit vectors $u \in \text{im}(P)$ with $u_i \geq 0$. Because $\Pt_i$ is scale invariant, it follows that $\Pt_i(P) = \Pt_i(\hat{P})$ and $$\Pt_i((I + xM)P) = \max_{u \in \hat{P}} \Pt_i(u + xMu).$$ 
For any $u \in \hat{P}$ we have 
\begin{equation*}
\begin{split}
\Pt_i(u + xMu) &= \frac{|(u + xMu)_i|}{\sum_{j \neq i} |(u + xMu)_j|} \\
&= \epsilon + \frac{|(u + xMu)_i| - \epsilon \sum_{j \neq i} |(u + xMu)_j|}{\sum_{j \neq i} |(u + xMu)_j|}. \\
\end{split}
\end{equation*}
Because $(q_i)_i = 1$, and $|(q_i)_j| \leq \epsilon$ for $j \neq i$, for any $x \in \R^n$ with $x_i \geq 0$, $|x_i| - \epsilon \sum_{j \neq i} |x_j| \leq \langle q_i, x \rangle.$ Furthermore, $q_i$ is orthogonal to all $u \in \hat{P}$. As a result,
\begin{equation*}
\begin{split}
\Pt_i(u + xMu) &\leq \epsilon + \frac{\langle q_i, u + xMu \rangle}{\sum_{j \neq i} |(u + xMu)_j|} \\
&= \epsilon + \frac{\langle q_i, xMu \rangle}{\sum_{j \neq i} |(u + xMu)_j|}.
\end{split}
\end{equation*}

Let $\hat{p_i} = \frac{p_i}{|p_i|}$, which implies $\hat{p_i} \in \hat{P}$. Our strategy is to relate $\Pt_i(u + xMu)$ to the distance from $u$ to $\hat{p_i}$. We have that
\begin{equation*}
\begin{split}
\Pt_i(u + xMu) &\leq \epsilon + \frac{\langle q_i, xMu \rangle}{\sum_{j \neq i} |(u + xMu)_j|} \\
&= \epsilon + \frac{x\langle q_i, M(u - \hat{p_i}) \rangle}{\sum_{j \neq i} |(u + xMu)_j|} + \frac{x\langle q_i, Mp_i \rangle}{\sum_{j \neq i} |(u + xMu)_j|}.
\end{split}
\end{equation*} 
Since $M$ is fixed and all $u \in \hat{P}$ are unit vectors, there exist $A, \geq 0$ and $B >0$ that are dependent on $M, P, Q$ and independent of $x$ such that $$\frac{\langle q_i, M(u - \hat{p_i}) \rangle}{\sum_{j \neq i} |(u + xMu)_j|} \leq A||u - \hat{p_i}||,$$ and $$\frac{\langle q_i, Mp_i \rangle}{\sum_{j \neq i} |(u + xMu)_j|} \leq B \langle q_i, Mp_i \rangle.$$ Therefore,
\begin{equation}\label{ABound}
\Pt_i(u + xMu) \leq \epsilon + Ax||u - \hat{p_i}|| + Bx\langle q_i, Mp_i \rangle.
\end{equation}
We now state a lemma about metric spaces and apply it to our particular case. A proof can be found in the appendix.

\begin{lemma}\label{topology}
Let $U$ be a compact metric space with metric $d$, and $\alpha: U \times \R^+ \to \R$ be a continuous function with the following properties.
\begin{enumerate}
	\item There exists a unique $u_0 \in U$ such that $\alpha(u_0, 0) > \alpha(u, 0)$ for all $u \neq u_0$.
	\item $\alpha(u, x) \leq C + Axd(u, u_0) + Bx$ for constants $A, B, C \in \R$ with $A \geq 0$ and $B \leq 0$.
\end{enumerate}
Then we have the following.
\begin{enumerate}
	\item If $B < 0$ then $\alpha(u, x) \leq C - \Theta(x)$ as $x \to 0$.
	\item If $B = 0$ then $\alpha(u,x) \leq C + o(x)$ as $x \to 0$.
\end{enumerate}
\end{lemma}\emph{ }\\
We now verify that the conditions of Lemma \ref{topology} are met. $\hat{P_i}$ is a compact metric space with standard euclidean distance metric. Taking $\alpha(u, x) = \Pt_i(u + xMu)$, we have the following.
\begin{enumerate}
	\item By Lemma \ref{uniquep}, $\hat{p_i}$ is the unique maximum of $\alpha(u, 0)$.
	\item By Equation \ref{ABound}, we have that $\alpha(u, x) \leq \epsilon + Axd(u, \hat{p_i}) + Bx\langle q_i, Mp_i \rangle$
\end{enumerate}
Since $B \geq 0$, $B \langle q_i, Mp_i \rangle \leq 0$ with equality holding when $\langle q_i, Mp_i \rangle = 0$. This completes all of the conditions, and a direct application of Lemma \ref{topology} finishes the proof. 
\end{proof}

Lemma \ref{opbounds} gives a sufficient condition on $M$ for perturbing by $M$ to allow $\Pt_i(P + xMP) \leq \Pt_i(P)$. Applying it simultaneously for all $i$ and noting that $P$ is a global optimum gives us the following.

\begin{lemma}\label{linprog}
Let $M$ be a matrix such that for all $1 \leq i \leq n$, $\langle q_i, Mp_i \rangle \leq 0$. Then for all $1 \leq i \leq n$, $\langle q_i, Mp_i \rangle = 0$. 
\end{lemma}

\begin{proof}

We assume towards a contradiction, that for some $i$, $\langle q_i, Mp_i \rangle <0$. Without loss of generality, $i=1$. By Lemma \ref{opbounds}, as $x \to 0$,
\begin{equation} \label{epsilonboundsLemma48}
\begin{split}
\Pt_1((I + xM)P) &\leq \epsilon -\Theta(x), \\
\Pt_2((I + xM)P) &\leq \epsilon + o(x), \\
\dots \\
\Pt_n((I + xM)P) &\leq \epsilon + o(x).
\end{split}
\end{equation}
Our strategy is to ``massage" the matrix $(I + xM)P$ into a matrix that has alignment strictly lower than $\epsilon$. 

Let $W = [w_1, w_2, \dots, w_n]$ be the matrix with columns $w_i$. For sufficiently small $x$, $(I + xM)$ is invertible. Let $w_i^x$ be the $i$th column of $W(1 + xM)^{-1}$.  Since $WP = 0$, we also have that 
\begin{equation*}
\begin{split}
0 &= WP \\
&= W(1 + xM)^{-1}(1 + xM)P.
\end{split}
\end{equation*}
It follows that the row space and column space of $W(1 + xM)^{-1}$ and $(1 + xM)P$ are orthogonal conjugates. Let $S^x = \{w_1^x, w_2^x, \dots, w_n^x\}$, and let $\epsilon_i^x = \Pt_i((1 + xM)P)$. By Theorem \ref{geopont}, $\alpha_i(S^x) = \epsilon_i^x$. 

If $x$ is sufficiently small, then $w_i^x$ is very close to $w_i$ for all $i$. As a result, $w_1^x, w_2^x, \dots, w_n^x, -w_1^x, \dots -w_n^x$  form a  convex polygon. Therefore, by the same argument given in Lemma \ref{uniquep}, $\epsilon_i^xw_i^x = w_i^x \cap \overline{w_{i-1}^xw_{i+1}^x}.$

\begin{figure}[t]
\centering
\begin{tikzpicture}
\node at (0, -0.3) {$O$};

\draw (-2.8, 0) -- (2.8,0);
\node at (3.2, 0.3) {$A$};
\node at (-2.9, -0.2) {$-w_1^x$};

\draw (2.5,1) -- (-2.5,-1);
\node at (2.9, 1.3) {$B$};
\node at (-2.6, -1.2) {$-w_2^x$};

\draw (1, 2.35) -- (-1, -2.35);
\node at (1.4, 2.7) {$C$};
\node at (-1.1, -2.55) {$-w_3^x$};

\node at (2.3, 1.3) {$D$};
\node at (1.7, 0.5) {$D'$};
\draw (1, 2.35) -- (2.4, 0);
\node at (2.4, -0.3) {$A'$};

\draw (2.8,0) -- (1, 2.35);

\node at (-0.5, 1) {$\dots$};
\node at (0.5, -1) {$\dots$};

\draw (-2.5, 1.3) -- (2.5, -1.3);
\node at (-2.4, 1.4) {$w_n^x$};
\node at (2.9, -1.2) {$-w_n^x$};

\end{tikzpicture}
\caption{Diagram for scaling $w_1^x$}\label{fig2}
\end{figure}
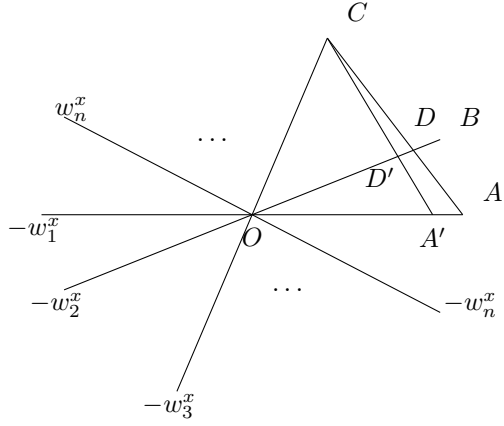

We now mimic the proof of Lemma \ref{uniform}. The key idea is to scale the lengths of $w_i^x$ so that the values of $\epsilon_i^x$ come closer together with the goal of making them all strictly smaller than $\epsilon$. This will contradict the optimality of $\epsilon$, finishing the proof. 

Fix $c_1 > 0$, and suppose we ``scale" $w_1^x$ to $w_1^x(1 - c_1x)$. Our goal is to find the effect that this has on $\epsilon_1^x, \epsilon_2^x, \dots, \epsilon_n^x$. To avoid confusion, we let $\gamma_i^x$ denote the new value of $\epsilon_i^x$ after scaling. 

Because $\epsilon_i^xw_i^x$ is the intersection of $\overline{w_{i-1}^xw_{i+1}^x}$ and $w_i^x$, $\gamma_i^x = \epsilon_i^x$ for $i \notin \{1, 2, n\}$. By a direct computation, $\gamma_1^x = \frac{\epsilon_1^x}{1 - c_1x}$. Provided $c_1$ is sufficiently small and applying equation \ref{epsilonboundsLemma48}, we have

\begin{equation*}
\begin{split}
\gamma_1^x &=\frac{\epsilon_1^x}{1 - c_1x} \\
 &= \epsilon_1^x(1 + c_1x + o(x)) \\
&\leq (\epsilon - \Theta(x))(1 + 2c_1x) \\
& < \epsilon.
\end{split}
\end{equation*}
 
Next, we bound $\gamma_2^x$. Refer to Figure \ref{fig2} for the following. Let $O$ be the origin, let $A, B, C$ denote $w_1^x, w_2^x, w_3^x$ respectively, and let $D = \overline{AC} \cap \overline{OB}$. Then, let $A' = w_1^x(1 - c_1x)$ and $D' = \overline{A'C} \cap \overline{OB}$. It follows that $$\frac{OD}{OB} = \epsilon_2^x, \text{ and }\frac{OD'}{OB} = \gamma_2^x.$$

Therefore, $\gamma_2^x = \frac{OD'}{OD}\epsilon_2^x.$ Because $x$ ranges over a sufficiently small neighborhood, there exists a constant $c_2$, determined only by $\{w_i: 1 \leq i \leq n\}$, such that $DD' \geq c_2AA'$. By the definition of $A'$, $\frac{OA'}{OA} = 1 - c_1x$, which means $AA' = |w_1^x|c_1x$. As a result,

\begin{equation*}
\begin{split}
\gamma_2^x &= \epsilon_2^x\frac{OD'}{OD}\\
&= \epsilon_2^x(1 - \frac{DD'}{|w_2^x|}) \\
&\leq \epsilon_2^x(1 - \frac{c_2AA'}{|w_2^x|}) \\
& = \epsilon_2^x(1 - \frac{c_1c_2x|w_1^x|}{|w_2^x|}).
\end{split}
\end{equation*}

However, by Equation \ref{epsilonboundsLemma48}, $\epsilon_2^x \leq \epsilon + o(x)$. Therefore, we have that 
\begin{equation*}
\begin{split}
\gamma_2^x &\leq \epsilon_2^x(1 - \frac{c_1c_2x|w_1^x|}{|w_2^x|}) \\
&\leq (\epsilon + o(x))(1 - \Theta(x)) \\
&\leq \epsilon - \Theta(x),
\end{split}
\end{equation*}
as $x$ goes to $0$. 

It is possible to apply the same argument for $\gamma_n^x$, but this won't be necessary, so we simply bound $\gamma_n^x \leq \epsilon_n^x$. In summary, this gives us the following upper bounds on $\gamma_i^x$. 
\begin{equation*}
\begin{split}
\gamma_1^x &< \epsilon,  \\
\gamma_2^x &\leq \epsilon - \Theta(x),\\
\gamma_3^x &= \epsilon_3^x \leq \epsilon + o(x), \\
&\dots \\
\gamma_n^x &< \epsilon_n^x \leq \epsilon + o(x).
\end{split}
\end{equation*}

We can repeat this process by scaling $w_2^x$ by $1 - c_3x$ for some sufficiently small constant $c_3$. If $x$ is sufficiently small, the bounding procedure we used for $\gamma_i^x$ will similarly work, as $c_3x$ still asymptotically dominates $o(x)$ by definition. Doing this results in $\gamma_2^x < \epsilon$ and $\gamma_3^x \leq \epsilon - \Theta(x)$. Repeating this procedure $n$ times yields yields a configuration in which all $\epsilon_i^x$ are strictly less than $\epsilon$, which contradicts the optimality of $\epsilon$. Therefore our initial assumption was false, and the claim must hold.

\end{proof}

\begin{lemma} \label{Lambda}
There exists diagonal matrix $\Lambda$ with positive diagonal elements such that $$P\Lambda Q^t = Q^tP\Lambda = 0.$$
\end{lemma}

\begin{proof}
Since $Q^tP = 0$, it suffices to show $P\Lambda Q^t = 0$. By Lemma \ref{linprog}, for any matrix $M$ such that 
\begin{equation*}
\begin{split}
\langle q_1, Mp_1 \rangle &\leq 0, \\
\langle q_2, Mp_2 \rangle &\leq 0,\\ 
\dots \\
\langle q_n, Mp_n \rangle &\leq 0,
\end{split}
\end{equation*}
$\langle q_i, Mp_i \rangle = 0$ for all $i$. Switching inner products, we have $\langle q_i, Mp_i \rangle = \langle q_ip_i^t, M \rangle$ where $\langle A, B \rangle$ is defined as $\text{Tr} A^tB$ for matrices $A,B$. We now apply the following variant of Farkas Lemma (proved in the appendix). 
\begin{lemma} \label{farkas}
Let $a_1, a_2, \dots, a_n \in \R^k$ be vectors such that for all $i$, there does not exist any $v_i \in \R^k$ with $\langle v_i, a_i \rangle < 0$ and $\langle v_i, a_j \rangle \leq 0$ for $j \neq i$.  Then there exist positive reals $\lambda_1, \lambda_2, \dots, \lambda_n$ such that $\sum \lambda_i a_i = 0$. 
\end{lemma}

Using this, there exist $\lambda_1, \lambda_2, \dots, \lambda_n > 0$ such that $\sum \lambda_i q_ip_i^t = 0.$ Letting $\Lambda$ have diagonal entries $\lambda_1, \lambda_2, \dots, \lambda_n$ implies $P\Lambda Q^t = 0$.
\end{proof}

We will continue to refer to the diagonal matrix found in Lemma \ref{Lambda} as $\Lambda$. 

\begin{lemma} \label{symmetric}
$P\Lambda$ is symmetric.
\end{lemma}

\begin{proof}
Right multiplication by $\Lambda$ scales the columns of $P$. By Lemma 13, $P\Lambda$ has form 
\[P\Lambda = \left( \begin{array}{cccc}
\lambda_1 & c_1 & & d_n\\
d_1 & \ddots & \ddots & \\
& \ddots & \ddots & c_{n-1} \\
c_n & & d_{n-1} & \lambda_n \end{array} \right) ,\]
for $c_i , d_i < 0$ for $1 \leq i \leq (n-1)$ and $c_n, d_n > 0$. Each column of $P$ has alignment $\epsilon$. Since $Q^tP = 0$ and $|Q_{ij}| \leq \epsilon$ for $i \neq j$,
\begin{enumerate}
	\item $Q_{ij} = \epsilon$ if $P_{ij} < 0$, $i \neq j$,
	\item $Q_{ij} = -\epsilon$ if $P_{ij} > 0$, $i \neq j$,
\end{enumerate}
By Lemma \ref{Lambda}, $Q^tP\Lambda = P\Lambda Q^t = 0$. It follows that,
\begin{equation*}
\begin{split}
\lambda_1 &= \epsilon(|c_1| + |d_n|) = \epsilon(|d_1| + |c_n|),\\
\lambda_2 &= \epsilon(|c_2| + |d_1|) = \epsilon(|d_2| + |c_1|),\\
\ldots \\
\lambda_n &= \epsilon(|c_n| + |d_{n-1}|) = \epsilon(|d_n| + |c_{n-1}|).\\
\end{split}
\end{equation*}
By rearranging the equations, we see that $$|c_1| - |d_1| = |c_2| - |d_2| = \dots = |c_n| - |d_n|.$$ Recall that $A$ is the space of linear dependencies of $S = \{w_1, w_2, \dots, w_n\}$. Therefore, $\sum (p_i)_j w_j = 0$. Substituting $c_i, d_i$ and multiplying by $\lambda_i$, we see that $$c_{i-1}w_{i-1} + d_iw_{i+1} = \lambda_iw_i.$$ Let $\phi_i$ denote the angle between $w_i$ and $w_{i+1}$. It follows that $$\frac{|d_i|}{|c_{i-1}|} = \frac{|w_{i-1}|\sin \phi_{i-1}}{|w_{i+1}|\sin \phi_{i}}.$$ Multiplying these inequalities over all $i$ implies $$\prod |d_i| = \prod |c_i|.$$ Because $|c_i| - |d_i|$ is the same for all $i$, this implies that $|c_i| = |d_i|$ for all $i$. Since $c_i, d_i$ have the same sign, we have that $c_i = d_i$ which implies that $P\Lambda$ is symmetric, as desired. 
\end{proof}

\begin{lemma} \label{rank}
$P\Lambda$ has rank $\geq n-2$
\end{lemma}

\begin{proof}
 If we delete the first column and the last row, we get a matrix that is lower diagonal with no zeroes on the diagonal (since $c_i \neq 0$ for all $i$). Thus the rank is at least $n-2$.
\end{proof}\emph{ }\\
Let $P\Lambda = S$. We make the following observations about $S$.
\begin{enumerate}
	\item By Lemma \ref{Lambda}, $Q^tS = SQ^t = 0$. Taking transposes, we have $QS = Q^tS = 0$.
	\item By Lemma \ref{symmetric} $S$ is symmetric.
	\item By Lemma \ref{rank}, $S$ has rank is at least $n-2$.
\end{enumerate}

Let $Q' = \frac{Q + Q^t}{2}$. Then $Q'S = 0$, which means that $Q'$ has rank at most $2$. Furthermore, all off-diagonal elements of $Q'$ are at most $\epsilon$ by absolute value. All that remains is to show that $Q'$ is positive semidefinite. 

$Q'$ has rank at most $2$, and therefore has at most two non-zero eigenvalues, and at least $n-2$ eigenvalues of $0$. Since $\epsilon \leq 1$, each eigenvalue of $Q'$ is at most $n$. Since $\text{tr}(Q') = n$, both non-zero eigenvalues of $Q'$ must be non-negative, which makes $Q'$ positive semidefinite.

This implies that $\SP(n,2) \leq \epsilon$, which gives $\SP(n, 2) = \ASP(n, 2) = \epsilon$. By citing that $\SP(n,2) = \cos \frac{\pi}{n}$ (\cite{conway} for example), we are done.

\bibliographystyle{ieeetr}
\bibliography{references}

\appendix

\section{Proofs of Lemmas}

\begin{proof}
(Lemma \ref{holder}) Let $\mathbb{L}_1(w) = \sum_1^m |w_i|$ for any $w \in \R^m$. Then $\epsilon \mathbb{L}_1$ is sub-linear on $\R^m$. Because $\epsilon\mathbb{L}_1$ dominates $f$, the Hahn Banach Theorem implies there exists $g: \R^m \to R$ such that $g(v) = f(v)$ for all $v \in V$, and $g$ is dominated by $\epsilon \mathbb{L}_1$. $g$ is a linear map, so there exists a unique $u \in \R^m$ such that $g(v) = \langle u, v \rangle$. This suffices. If $u_i > \epsilon$, then we have $g(e_i) = \langle u_i, e_i \rangle > \epsilon \mathbb{L}_1(e_i)$, a contradiction. 
\end{proof}

\begin{proof}

(Lemma \ref{topology}) We first show the following lemma that we use in both cases.

\begin{lemma}\label{lem_append}
Let $M \subset U$ be a compact subspace such that $u_0 \notin M$. Then there exists $\delta > 0$ and $D < C$  such that for all $x < \delta$ and $m \in M$, $\alpha(m, x) < D$.
\end{lemma}

\begin{proof}
Let $C' = \max_{u \in M} \alpha(u, 0)$. Then $C' < C$ because $u_0$ is the unique optimum of $\alpha(u,0)$ over $U$.

Let $f(x) = \max_{m \in M} \alpha(m, x).$ Then $f(0) = C'$.  Because $M$ is compact, $f$ is a continuous function. Let $D = \frac{C + C'}{2}$. Then there exists $\delta > 0$ such that $f(x) < C''$ for $x < \delta$. Therefore, for all $m \in M$ and $x < \delta$, $\alpha(m,x) \leq f(x) < D$ as desired. 
\end{proof}

We split into two cases corresponding to the statement of the theorem.\\
\textbf{Case 1: $B < 0$.}

Pick some small open neighborhood $N$ of $u_0$ such that for all $u \in N$, $Ad(u, u_0) < \frac{1}{2}|B|$. Then for all $u \in N$, $\alpha(u, x) \leq C + \frac{1}{2}Bx = C - \Theta(x)$. Therefore, we need only show that the same holds for $u \notin N$.

Let $M = U - N$. By Lemma \ref{lem_append}, there exists $\delta > 0$, and $D < C$ such that $\alpha(m,x) < D$ for $x < \delta$. Therefore $\alpha(m,x) < C - (C-D) < C - cx$ for a sufficiently small constant $c$. The result follows. \\
\textbf{Case 2: $B = 0$.}

Let $N_x \subset U$ denote all $u$ such that $\alpha(u, x) > C$. Define $r_x = \sup_{u \in N_x} d(u_0, u).$ Therefore, for all $u \in U$, $\alpha(u,x) \leq C + Axr_x.$ It suffices to show that as $x \to 0$, $r_x \to 0$. 

Let $M_\epsilon = \{u, d(u, u_o) \geq \epsilon\}$. $M_\epsilon$ is compact and doesn't contain $u_0$. By Lemma \ref{lem_append}, there exists $\delta_\epsilon > 0$ such that for all $0 < x < \delta_\epsilon$ and $m \in M_\epsilon$, $\alpha(m, x) < C$. It follows that $r_x \leq \epsilon$ for all $x < \delta$. Since $\epsilon$ was arbitrary,  $r_x \to 0$ as $x \to 0$ and we are done.
\end{proof}

\begin{proof}
(Lemma \ref{farkas}) Let $A$ be the matrix with columns $a_1, a_2, \dots a_n$. For all $a_i$, there does not exist any $v$ with $A^tv \leq 0$ and $a_i^tv < 0$. Therefore, by Farkas lemma, there exists $x_i \geq 0$ such that $Ax_i = -a_i$. Letting $x_i' = x_i + e_i$ where $e_i$ is the vector with $i$th coordinate $1$ and all other coordinates $0$, we see that $\lambda = \sum_1^n x_i'$ satisfies $A\lambda = 0$. Furthermore, all coordinates of $\lambda$ are strictly positive (in fact at least $1$) which completes the proof. 
\end{proof}

\end{document}